\documentclass[a4paper,12pt]{article}

\usepackage{amsmath}
\usepackage{amssymb}
\usepackage{amsthm}
\usepackage{mathtools}
\usepackage{graphicx}
\usepackage{psfrag}
\usepackage{subfig}
\usepackage{authblk}
\usepackage{xcolor} 
\usepackage{url} 
\usepackage{enumerate} 
\usepackage{xtab,booktabs}
\usepackage[normalem]{ulem}
\usepackage{tikz}

\usepackage[hmargin=2.5cm,vmargin=2.5cm]{geometry}

\usepackage{enumitem}

\newtheorem{theo}{Theorem}
\numberwithin{theo}{section} 

\newtheorem{lem}[theo]{Lemma}

\newtheorem{cor}[theo]{Corollary}

\newtheorem{ass}[theo]{Assumption}
\newtheorem{cl}[theo]{Claim}

\newcommand{\R}{\mathbb R}

\newcommand{\sgn}{\operatorname{sgn}}

\usepackage{bm}

\begin{document}
\begin{center}

\section*{H\"{o}lder regularity for the trajectories of generalized charged particles in 1D}

\large{Thomas G. de Jong\\}
\small{Faculty of Mathematics and Physics, 
Institute of Science and Engineering, \\
Kanazawa University,
Kanazawa, Japan. \\
{\tt t.g.de.jong.math@gmail.com}}
\\[2mm]

\large{Patrick van Meurs\\}
\small{Faculty of Mathematics and Physics, 
Institute of Science and Engineering, \\
Kanazawa University,
Kanazawa, Japan.\\
{\tt pjpvmeurs@staff.kanazawa-u.ac.jp}}
\\[2mm]

\end{center}

\begin{abstract}
\noindent We prove H\"older regularity for the trajectories of an interacting particle system. The particle velocities are given by the nonlocal and singular interactions with the other particles. Particle collisions occur in finite time. Prior to collisions the particle velocities become unbounded, and thus the trajectories fail to be of class $C^1$. Our H\"older-regularity result supplements earlier studies on the well-posedness of the particle system which imply only continuity of the trajectories. Moreover, it extends and unifies several of the previously obtained estimates on the trajectories. Our proof method relies on standard ODE techniques: we transform the system into different variables to expose and exploit the hidden monotonicity properties.
\end{abstract}
\noindent {\bf Keywords:} Interacting particle systems, ODEs with singularities, regularity. \\ 
{\bf MSC:} 34E18, 74H30. 


%


\section{Introduction}
\label{s:intro}

We are interested in improving the regularity properties of the solution to a hybrid system of ODEs (see \eqref{Pn} below) which was recently proven to be well-posed \cite{VanMeursPeletierPozar22,VanMeurs23ArXiv}. This system appears in plasticity theory as a model for the dynamics of crystallographic defects; see \cite{VanMeurs23ArXiv} and the references therein.

\subsection{The governing equations formally}

The system is an interacting particle system. It is formally given by
   \begin{equation} \label{Pn}
   \left\{ \begin{aligned}
     &\frac{dx_i}{dt}  
  = \sum_{j \neq i} b_i b_j f (x_i - x_j) + b_i g(x_i)
     && t \in (0,T), \ i = 1,\ldots, n \\
     &+ \text{annihilation upon collision,}
     &&
   \end{aligned} \right.
\end{equation}
where $n \geq 1$ is the number of particles, $x_1 \leq \ldots \leq x_n$ are the time-dependent particle positions, and $b_1, \ldots, b_n \in \{-1,+1\}$ are the fixed signs of the particles. The given function $g : \R \to \R$ is an externally applied force, and the given odd function $f : \R \setminus \{0\} \to \R$ describes the interaction force between each pair of two particles. The interaction force is such that particles of the same sign repel and particles of opposite sign attract. Figure \ref{fig:fg} illustrates typical choices of $f$ and $g$; Assumption \ref{a:fg} below lists the properties we impose on them.

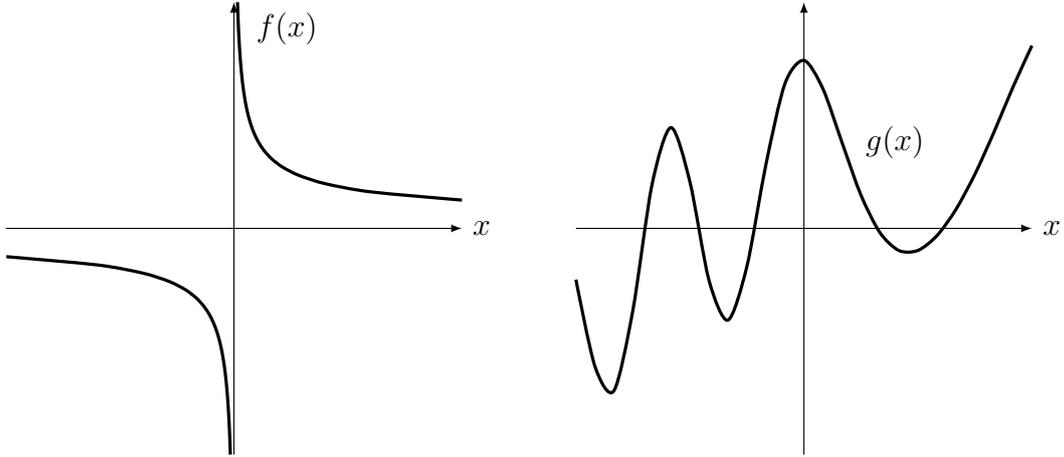
\begin{figure}[h]
\centering
\begin{tikzpicture}[scale=1.5, >= latex]    
\def \w {2}
        
\draw[->] (0,-\w) -- (0,\w);
\draw[->] (-\w,0) -- (\w,0) node[right] {$x$};
\draw[domain=0.25:\w, smooth, very thick] plot ({ 1 / (8*\x*\x) }, \x);
\draw[domain=-\w:-0.25, smooth, very thick] plot ({ -1 / (8*\x*\x) }, \x);
\draw (0.1, \w) node[anchor = north west]{$f(x)$};

\begin{scope}[shift={(2.5*\w, 0)},scale=1]
\draw[->] (0,-\w) -- (0, \w);
\draw[->] (-\w,0) -- (\w,0) node[right] {$x$};
\draw[domain=-\w:\w, smooth, very thick] plot ( \x, { .5 + sin(240*\x + 100*cos(69*\x)) + .4*\x - .1*\x*\x } );
\draw (.8, .5*\w) node[below]{$g(x)$};
\end{scope}
\end{tikzpicture} \\
\caption{Examples of $f$ and $g$.}
\label{fig:fg}
\end{figure}

The feature which makes \eqref{Pn} interesting is that $f$ is singular at $0$. As a result, particles of opposite sign collide in finite time and with unbounded velocity. Figure \ref{fig:trajs} sketches the particle dynamics including several collisions. 
The maximal time of existence of solutions to the system of ODEs is the first collision time $\tau_1$. To extend it beyond collisions, the following collision rule is applied: whenever a pair of particles with opposite charge collide, they are removed (annihilated) from the system. The removal of pairs is done sequentially. As a consequence, not all colliding particles are necessarily removed; see Figure  \ref{fig:trajs}. The surviving particles continue to evolve by the system of ODEs \eqref{Pn}. We refer to the combination of the system of ODEs and the annihilation rule in \eqref{Pn} as a hybrid system of ODEs.

\begin{figure}[ht]
\centering
\begin{tikzpicture}[scale=1.5, >= latex]
\def \sqtwo {1.414}
\def \rr {0.03}

\draw[->] (0,0) -- (0,3.5) node[above] {$t$};
\draw[->] (-2,0) -- (4.8,0) node[right] {$x$};

\draw[dotted] (0,.5) node[left]{$\tau_1$} -- (3.581,.5);
\draw[dotted] (0,1.5) node[right]{$\tau_2$} -- (-1,1.5);
\draw[dotted] (0,3) node[left]{$\tau_3$} -- (2,3);

\draw (-1.7,0) node[below] {$x_1$};
\draw (-.3,0) node[below] {$x_2$};
\draw (.3,0) node[below] {$x_3$};
\draw (.8,0) node[below] {$x_4$};
\draw (1.6,0) node[below] {$x_5$};
\draw (2.1,0) node[below] {$x_6$};
\draw (3,0) node[below] {$x_7$};
\draw (3.7,0) node[below] {$x_8$};
\draw (4.4,0) node[below] {$x_9$};

\begin{scope}[shift={(2,3)},scale=1] 
    \draw[thick, red] (0,0) -- (0,.5);
    \draw[thick, blue] (0,-1.5) -- (0,0);
    \draw[domain=-1.732:1.581, smooth, thick, red] plot (\x,{-\x*\x});           
    \draw[domain=1.581:1.732, smooth, thick, blue] plot (\x,{-\x*\x});           
    \fill[black] (0,0) circle (\rr); 
\end{scope}

\begin{scope}[shift={(2,1.5)},scale=1, rotate = 270] 
    \draw[domain=0:1, smooth, thick, blue] plot (\x,{.08*\x*sqrt(\x)}); 
    \draw[domain=1:1.5, smooth, thick, blue] plot (\x,{.08*\x*sqrt(\x) - .1*(\x - 1)*sqrt(\x - 1)});         
\end{scope}

\begin{scope}[shift={(3.581,.5)},scale=1] 
    \draw[domain=-.707:.707, smooth, thick, red] plot ({2*exp(-\x/2)-2},{-\x*\x});
    \fill[black] (0,0) circle (\rr);  
\end{scope}

\begin{scope}[shift={(1.2,.5)},scale=1] 
    \draw[domain=0:.408, smooth, thick, red] plot (\x,{-3*\x*\x});
    \draw[domain=-.408:0, smooth, thick, blue] plot (\x,{-3*\x*\x});     
    \fill[black] (0,0) circle (\rr);     
\end{scope}

\begin{scope}[shift={(-1,1.5)},scale=1] 
    \draw[domain=-.707:0, smooth, thick, blue] plot (\x,{-3*\x*\x});  
    \draw[domain=0:.707, smooth, thick, red] plot (\x,{-3*\x*\x}); 
    \fill[black] (0,0) circle (\rr);         
\end{scope}

\end{tikzpicture} \\
\caption{A sketch of solution trajectories to \eqref{Pn}. Trajectories of particles with positive charge are colored red; those with negative charge blue. The black dots indicate the collision points.}
\label{fig:trajs}
\end{figure}
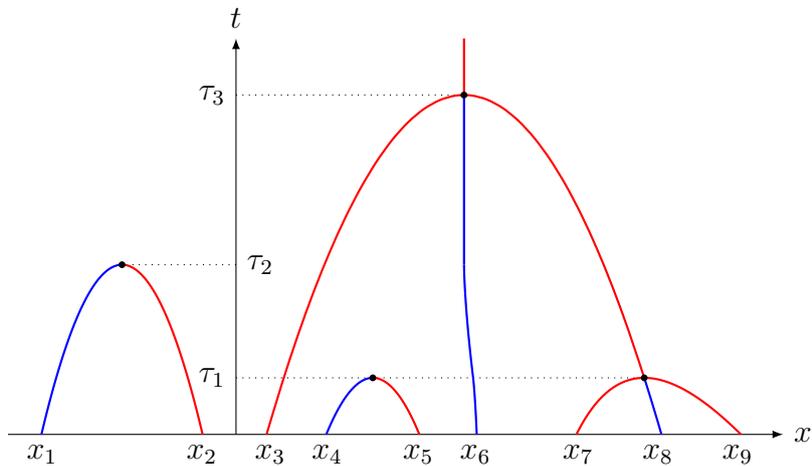

\subsection{Known results}
\label{s:intro:known}

\cite{VanMeursPeletierPozar22,VanMeurs23ArXiv} establish well-posedness of \eqref{Pn}, derive estimates on the particle trajectories near collisions, and pass to the limit $n \to \infty$ in the rescaled time $t' = t/n$. \cite{VanMeursPeletierPozar22} does this for the specific case of electrically charged particles (precisely, with $f(x) = \frac1x$ and $g \equiv 0$), and \cite{VanMeurs23ArXiv} generalized this to a any $f$ and $g$ that satisfy a slightly weaker set of assumptions than Assumption \ref{a:fg} (we comment on this in Section \ref{s:intro:disc}). In this paper, we make the following standing assumptions on $f$ and $g$:

\begin{ass} \label{a:fg}
$g : \R \to \R$ is Lipschitz continuous and $f: \R \setminus \{0\} \to \R$ satisfies:
\begin{enumerate}[label=(\roman*)]
  \item \label{a:fg:odd} $f$ is odd,
  \item \label{a:fg:sing} (singularity) there exists $a > 0$ such that \[
f = f_a + f_{\rm reg}, \qquad f_a(x) := \frac{\sgn(x)}{|x|^a}
\]
for some $f_{\rm reg} \in W_{\rm loc}^{2,1}(\R \setminus \{0\})$ with left limits $f_{\rm reg} (0+), f_{\rm reg}' (0+) \in \R$,
  \item \label{a:fg:mon} (monotonicity) on $(0,\infty)$, $f \geq 0$, $f' \leq 0$ and $f'' \geq 0$.
\end{enumerate}
\end{ass}

Next we briefly recall the results and arguments thereof given in \cite{VanMeursPeletierPozar22,VanMeurs23ArXiv}. Standard ODE theory provides the existence and uniqueness of solutions up to the first collision time $\tau_1$. The difficulty for obtaining global well-posedness of \eqref{Pn} is to get across $\tau_1$ and later collision times. This was done by showing that at each collision time $\tau$:
\begin{enumerate}
  \item the limits $x_i(\tau-) := \lim_{t \uparrow \tau} x_i(t)$ exist, and
  \item all particles that collide at $\tau$ at the same point $y \in \R$ must have \textit{alternating} signs before $\tau$.
\end{enumerate}
The second statement implies that the annihilation rule leaves either no particles at $y$ (when the number of colliding particles is even) or precisely one particle at $y$; see Figure \ref{fig:trajs}. In both cases the ODE system can be restarted from $\tau$ after removing the annihilated particles. Iterating this procedure over all the collision times $\tau_1, \ldots, \tau_K$ (note that $K \leq \frac n2$), a global solution to \eqref{Pn} is constructed.  By this construction, the particle trajectories on any interval of subsequent collision times $[\tau_{k-1}, \tau_{k}]$ (including $[\tau_0, \tau_1]$ and $[\tau_K, \tau_{K+1}]$ with $\tau_0 := 0 $ and $\tau_{K+1} := T$)  satisfy
\begin{equation} \label{regy:apri}
  x_i \in C([\tau_{k-1}, \tau_{k}]) \cap C^1([\tau_{k-1}, \tau_{k})) \qquad \text{for all } 1 \leq k \leq K+1.
\end{equation}

We refrain from providing a precise definition for the solution to \eqref{Pn}. We refer to \cite[Definition 2.5]{VanMeurs23ArXiv} for such a definition. The definition is technical solely because of bookkeeping reasons; the definition keeps track of the annihilated particles, and extends their trajectories beyond their times of annihilation. Here, we simply remark that the solution is unique up to a relabelling of the particles (see, e.g.\ Figure \ref{fig:trajs}, where at the three-particle collisions there is a choice which of the two positive particles survives). Furthermore, since the solution is iteratively constructed over the collision times, it is sufficient to focus on the properties of the solution before and at $\tau_1$.

\cite{VanMeursPeletierPozar22,VanMeurs23ArXiv,PatriziValdinoci17} establish several estimates of the trajectories around $\tau_1$. Before stating them, consider the following example: $n=2$, $b_1 b_2 = -1$, $g \equiv 0$ and $f(x) = \sgn (x) |x|^{-a}$ (i.e.\ $f_{\rm reg} = 0$). Then, the solution to \eqref{Pn} can be computed explicitly, and is given by 
\begin{align} \notag
  x_2(t) + x_1(t) 
  &= x_2(0) + x_1(0), \\ \notag 
  x_2(t) - x_1(t) 
  &= c_a (\tau_1 - t)^{\tfrac1{1+a}},
\end{align}
for some explicit constants $\tau_1, c_a > 0$. The  estimates in \cite{VanMeurs23ArXiv,PatriziValdinoci17} demonstrate to which extend this power law behavior extends to the general setting. In \cite[Proposition 3.4]{PatriziValdinoci17} (with $f_{\rm reg} = 0$) it is proven that 
\begin{equation} \label{ri:min}
  \min_{1 \leq i \leq n-1} x_{i+1}(t) - x_i(t)
\end{equation}
is H\"older continuous in $t$ with exponent $\frac1{1+a}$. In \cite{VanMeurs23ArXiv} it is shown that there exist constants $C,c > 0$ such that when particles $k, k+1,\ldots,\ell$ collide at the same point,
\begin{equation} \label{UB:ri}
  r_i(t) := x_{i+1} (t) - x_i(t) \leq C (\tau_1 - t)^{\tfrac1{1+a}}
  \qquad \text{for all } i = k,\ldots,\ell-1,
\end{equation}
\begin{equation} \label{LB:outer:xi}
  x_k(t) - x_k(\tau_1) \leq - c (\tau_1 - t)^{\tfrac1{1+a}}
  \quad \text{and} \quad
  x_\ell(t) - x_\ell(\tau_1) \geq c (\tau_1 - t)^{\tfrac1{1+a}}.
\end{equation}
for all $t < \tau_1$ close enough to $\tau_1$.

\subsection{Main result}

The properties cited above raise the question whether $x_i$ is H\"older continuous with exponent $\frac1{1+a}$. Note that by themselves these properties do not imply any H\"older regularity; to see this, consider for instance adding fast oscillations to the trajectories. This is not just a small technicality that was missed in the previous papers; the possibility of fast oscillations  cannot be excluded from the arguments in the corresponding proofs.

Our main result states that the particle trajectories are indeed H\"older continuity with exponent $\frac1{1+a}$: 

\begin{theo}[Main result] \label{theo:main} Let $n \geq 1$, $T > 0$, $b_1, \ldots, b_n \in \{-1,+1\}$, $f,g$ satisfy Assumption \ref{a:fg} and $x_1^\circ < \ldots < x_n^\circ$ be initial conditions of the particles. For the solution of \eqref{Pn} starting at $x_1^\circ, \ldots, x_n^\circ$, let $\tau_1, \ldots, \tau_K < T$ be the collision times. Set $\tau_0 = 0$ and $\tau_{K+1} = T$. Then, for each $1 \leq k \leq K+1$, each particle trajectory on $[\tau_{k-1}, \tau_k]$ is H\"older continuous with exponent $\frac1{1+a}$. 
\end{theo}

\subsection{Classical approach to dealing with singularities}

A classical approach to dealing with singular ODEs is to introduce new dependent and independent variables such that the vector field of the transformed system is locally continuous \cite{arnold1992ordinary, heggie1974global}. Typically, in the new variables the system has an equilibrium such that solutions of interest are contained on an invariant manifold induced by the dynamics in a neighborhood of that equilibrium; see \cite{JON2021topological,de2022uniqueness,de2020fungal,diacu1992regularization,bianchini2011invariant} for examples. Applying such a transformation to \eqref{Pn} on $[0,\tau_1]$, the equilibrium is given by the particle positions at $\tau_1$, which will be asymptotically reached when the transformed time $t'$ tends to $\infty$. More specifically, in the case $f(x) = \frac1x$ we can proceed by using a similar approach as the Kustaanheimo-Stiefel (KS) transformation \cite{kust1964regu,kust1965regu}. This transformation maps the independent variables to a higher dimensional space in which a simple change of independent variable leads to regularization of the equations. Finding the equilibria of these regularized equations is equivalent to a polynomial root-finding problem, and the problem of obtaining existence of the invariant manifold reduces to an eigenvalue problem. Unfortunately, these problems are already difficult to solve analytically for $n \geq 5$. Yet, for numerical studies these regularized variables provide a promising alternative approach.

\subsection{Idea of the proof}
First, we show that without loss of generality we can treat each collision independently from the other particles. Essentially, this means that we may assume that all particles collide at $x=0$ at $\tau_1$. Second, for each collision, we show that it is sufficient to supplement \eqref{LB:outer:xi} with a lower bound of the type 
\begin{equation} \label{ri:LB}
  r_i(t) \geq c_1 (\tau_1 - t)^{1/(1+a)} \qquad \text{for all } 1 \leq i \leq n-1. 
\end{equation}
The first step for showing \eqref{ri:LB} is to show that no two neighboring particles $x_i$ and $x_{i+1}$ can be too close together with respect to the distance to their other neighbors $x_{i-1}$ and $x_{i+2}$, as otherwise $x_i$ and $ x_{i+1}$ would collide first. This translates to a bound of the type 
$$r_i \geq c_2 \min \{ r_{i-1}, r_{i+1} \}. $$
This bound on itself is not enough to obtain \eqref{ri:LB}, because it still allows for a configuration in which $r_i, r_{i+1} \ll r_{i-1}, r_{i+2}$.  However, in such a situation, either $x_i$ and $x_{i+1}$ collide first, or $x_{i+1}$ and $x_{i+2}$ collide first, which would contradict that all particles collide at the same time. Generalizing this we obtain a bound of the type
$$ \sum_{k=i}^{j-1} r_k = x_j - x_i \geq c_3 \min \{ r_{i-1}, r_j \} $$
for all $i < j$. This is the key Lemma \ref{lem:rjimin}. From it and \eqref{LB:outer:xi} we derive the desired bound \eqref{ri:LB} in Lemma \ref{lem:main} by an iteration argument. The proof of Lemma \ref{lem:rjimin} relies on several quantitative bounds (see Lemmas \ref{lem:ri}, \ref{lem:rjiodd} and \ref{lem:rjieven}) which are inspired by the proofs in \cite{VanMeursPeletierPozar22} and \cite{VanMeurs23ArXiv}.

\subsection{Discussion} 
\label{s:intro:disc}

Our main result, Theorem \ref{theo:main}, generalizes and unifies in a clean, compact statement several previously obtained results such as the H\"older continuity of \eqref{ri:min} and the estimates in \eqref{UB:ri}.

Finally, we compare Assumption \ref{a:fg} on $f,g$ to the assumptions made in \cite[Assumption 2.2 and Theorem 2.7]{VanMeurs23ArXiv}. Assumption \ref{a:fg} is slightly stronger; Assumption \ref{a:fg}\ref{a:fg:sing} is replaced in \cite{VanMeurs23ArXiv} by the weaker assumption that $c x^{-a} \leq f(x) \leq C x^{-a}$ for some constants $c,C > 0$ and for all $x > 0$ small enough. We choose to impose Assumption \ref{a:fg}\ref{a:fg:sing} because it covers all applications for \eqref{Pn} given in \cite{VanMeurs23ArXiv} and it simplifies the proofs and the constants that appear in them.

\subsection{Overview}

The paper is organized as follows. In Section \ref{cl} we show that for proving Theorem \ref{theo:main} it is sufficient to zoom in on the trajectories at each collision separately. In Section \ref{s:rji} we exploit the monotonicity properties of $f$ to strongly reduce the full dependence between the equations of the system of ODEs in \eqref{Pn} at the cost of differential inequalities rather than equalities. Based on those inequalities we prove Theorem \ref{theo:main} in Section \ref{s:pf}.

\section{Reduction to a single collision}

In this section we show that it is sufficient to prove Theorem \ref{theo:main} for a single collision. Claim \ref{cl} states this in full detail.

\begin{cl} \label{cl}
Without loss of generality we may assume in Theorem \ref{theo:main} that:
\begin{enumerate}[label=(\roman*)]
  \item \label{cl:1} $k=1$, 
  \item \label{cl:2} $b_i = (-1)^i$ for all $1 \leq i \leq n$, 
  \item \label{cl:3} $f_{\rm reg} = 0$,
  \item \label{cl:4} $x_1^\circ, \ldots, x_n^\circ$ are such that $x_i(\tau) = 0$ for all $1 \leq i \leq n$ where $\tau \in (0, \tau_1]$ can be chosen freely, 
  \item \label{cl:5} $x_1 < \ldots < x_n$ on $[0,\tau)$, and
  \item \label{cl:6} there exist functions $F_i \in C([0,\tau])$ such that
\begin{equation} \label{ODE:xi}
  \frac{dx_i}{dt} = \sum^{n}_{ \substack{ j=1 \\ j \neq i } } b_i b_j f(x_i - x_j) + F_i(t)
  \qquad \text{for all } t \in (0, \tau) \text{ and all } 1 \leq i \leq n.
\end{equation}
\end{enumerate}
\end{cl}

\begin{proof} 
Let the setting in Theorem \ref{theo:main} be given. Since the system \eqref{Pn} is essentially the same on each time interval of subsequent collision times, it is sufficient to prove the H\"older regularity on $[0,\tau_1]$. Since the particle trajectories are continuous on $[0,\tau_1]$, they are by definition of $\tau_1$ separated on $[0,\tau_1)$. This separation has two consequences. First, since the particles are initially strictly ordered, they remain strictly ordered on $[0,\tau_1)$. Second, if two particles $x_i$ and $x_j$ do not collide at $\tau_1$, then their trajectories are separated by a positive distance on $[0,\tau_1]$. By grouping together particles that collide at the same point at time $\tau_1$, we can split \eqref{Pn} up into systems of ODEs of the form
\begin{equation} \label{PnI}
  \frac{dx_i}{dt} = \sum_{ \substack{ j \in I \\ j \neq i } } b_i b_j f_a(x_i - x_j) + F_i^I (t) \qquad i \in I, 
\end{equation}
where $I \subset \{1, \ldots, n\}$ contains the indices of the colliding particles (note that $I = \{k, \ldots, \ell\}$ for some $1 \leq k \leq \ell \leq n$ and that $I$ may be a singleton),  
\[
  F_i^I(t) := \sum_{ \substack{ j \in I \\ j \neq i } } b_i b_j f_{\rm reg}(x_i - x_j) + \sum_{ j \notin I } b_i b_j f(x_i - x_j) + b_i g(x_i)
\]
is continuous and bounded on $[0, \tau_1]$, and $b_i b_j = (-1)^{i+j}$ thanks to the `alternating signs' result mentioned in Section \ref{s:intro:known}. 

By considering $F_i^I(t)$ as a generic function in $C([0,\tau_1])$, the systems \eqref{PnI} indexed by $I$ decouple. Hence, we may focus on a single, arbitrarily chosen system. Note that \eqref{PnI} is invariant under swapping all signs and invariant in shifting space. Hence, we may assume that all particles in \eqref{PnI} collide at $y=0$. Finally, thanks to the known regularity in \eqref{regy:apri} it is sufficient to prove the H\"older regularity only on $[\tau_1 - \tau, \tau_1]$, where $\tau \in (0, \tau_1)$ can be chosen freely. 

From this construction Claim \ref{cl} follows. Indeed, by shifting the time variable to $t' = t - \tau_1 + \tau$, we obtain \ref{cl:4}. By relabeling the particles, we can transform $I = \{k, \ldots, \ell\}$ to $\{1, \ldots, n' \}$ with $n' = \ell - k + 1 \in \{1,\ldots,n\}$; this demonstrates \ref{cl:6}.
\end{proof}

\section{Estimates for distance between particles}
\label{s:rji}

In this section we start from the setting in Claim \ref{cl}. In particular, we consider system \eqref{ODE:xi}. We establish key estimates on the distance between particles on which our proof of Theorem \ref{theo:main} in Section \ref{s:pf} relies. For the particle distances we introduce the notation:
\begin{align*}
r_{ji} &:= r_{j,i} := x_{j} - x_i \in \R && 1 \leq i, j \leq n, \\
r_{i} &:= r_{i+1,i} = x_{i+1} - x_i > 0 && 1 \leq i <  n.    
\end{align*}
Note that $\sgn r_{ji} = \sgn (j-i)$.

We note that while we may assume $f_{\rm reg} = 0$, the estimates in this section also hold without this assumption provided that all three inequalities in Assumption \ref{a:fg}\ref{a:fg:mon} on the monotonicity are strict inequalities.

To obtain estimates on $r_{ji}$, we start in Section \ref{s:rji:ri} with examining the easier case of $r_i$. Afterwards, in Section \ref{s:rji:rji} we treat the general case. We demonstrate how particle interactions can be grouped together such that they give a positive or negative contribution to $\dot r_{ji}$. Finally, in Section \ref{s:rji:dotrji} we derive bounds on $\dot r_{ji}$.

\subsection{Notation}
We reserve $C, c > 0$ to denote generic positive constants which do not depend on the relevant parameters and variables. We think of $C$ as possibly large and $c$ as possibly small. The values of $C,c$ may change from line to line, but when they appear multiple times in the same display their values remain the same. When more than one generic constant appears in the same display, we use $C', C'', c'$ etc.\ to distinguish them. When we need to keep track of certain constants, we use $C_0, C_1$ etc.

\subsection{Distance between neighboring particles}
\label{s:rji:ri}

From the  governing equations \eqref{ODE:xi} we obtain the following ODE for each $r_i$:
\begin{align}
\dot{r}_{i} &= - 2 f(r_i) + G_i(t) - \sum_{ \substack{ k = 1 \\ k \neq i, i+1 }}^n [ b_i b_k  g(r_{ik}; r_i)]  , 
\label{eq:dotri}
\end{align}
where the first term in the right-hand side accounts for the interaction between $x_i$ and $x_{i+1}$, the second term 
$$G_i(t) := F_{i+1}(t) - F_{i}(t)$$ 
describes the external forcing and the summand accounts for the effect of particle $x_k$, where for a given parameter $\rho > 0$,
\begin{equation*}
  g(\cdot ;\rho) : \R \setminus \{-\rho,0\} \to \R, \qquad g(r; \rho) := f(\rho + r) + f(r).
\end{equation*}

\begin{figure}[ht]
\begin{center}
\includegraphics[width=8cm]{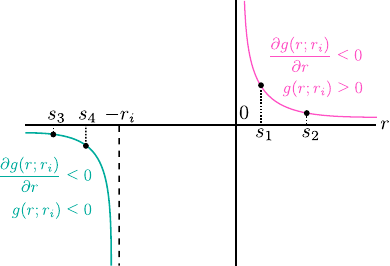}
\caption{Properties of $g(r;r_i)$ in \eqref{eq:dotri} for $r_i > 0$ fixed. The distance $r_{ik}$ represented by $r$ is contained in  $(-\infty,-r_i)$ or  $(0, \infty)$. The inequalities follow from the monotonicity properties of $f$. \label{fig:g}} 
\end{center} 
\end{figure}

Next we investigate how the presence of particles and particle pairs contribute to the sum in \eqref{eq:dotri}. In Figure \ref{fig:g} we present monotonicity properties of $g(\cdot; r_i)$. For convenience we assume that $b_i = 1$ (note that $b_{i+1} = -1$). The sign of $g$ implies that a positive particle $x_k$ to the left (i.e.\ $r_{ik} > 0$) tends to decrease $r_i$, i.e.\ its contribution to the sum in \eqref{eq:dotri} is negative. More generally, whether a particle $x_k$ tends to increase or decrease $r_i$ depends on $b_k$ and whether $x_k$ is to the left or to the right of the pair $(x_i, x_{i+1})$, but not on its distance to $(x_i, x_{i+1})$. In Table \ref{table:gandh} below we list all four cases.

Next we consider a pair of neighboring particles $(x_k, x_{k+1})$ to either the left of $x_i$ or to the right of $x_{i+1}$. Observe from Assumption \ref{a:fg}\ref{a:fg:mon} (monotonicity)  that for any $r_i > 0$:
\begin{enumerate}
  \item for all $0 < s_1 < s_2$ we have $g (s_2; r_i) - g(s_1; r_i) < 0$, and 
  \item for all $s_3 < s_4 < -r_i$ we have $g (s_4; r_i) - g(s_3; r_i) < 0$. 
   \end{enumerate}
The above properties are visualized in Figure \ref{fig:g}. Hence, the pair $(x_k, x_{k+1})$ tends to either increase or decrease $r_i$ independently of the positions of $x_k$ and $x_{k+1}$; see Table \ref{table:gandh} below. In \eqref{eq:dotri} such a pair corresponds to two consecutive terms of the sum.

\subsection{Distance between any pair of particles}
\label{s:rji:rji}

Next we consider $r_{ji}$ with $j >i$. Similar to the computation leading to \eqref{eq:dotri}, we obtain
\begin{align}
\dot{r}_{ji} &=\sum_{k=i}^{j-1} b_j b_k f(r_{jk}) - \sum_{k=i+1}^{j} b_i b_k f(r_{ik})  + \sum_{k\leq i-1} \left[b_j b_k f(r_{jk}) - b_i b_k f(r_{ik}) \right] \nonumber \\
& \qquad +  \sum_{k\geq j+1} \left[ b_j b_k f(r_{jk}) -  b_i b_k f(r_{ik}) \right] + G_{ji}(t), \label{eq:dotrji}
\end{align}
where 
$$
  G_{ji}(t) := F_j(t) - F_i(t). 
$$
Note that $\sgn f(r_{k \ell}) = \sgn (k - \ell)$ for all $k \neq \ell$.
The first two sums in the right-hand side generalize the first term in the right-hand side of \eqref{eq:dotri}; not only do they account for the interaction between $x_i$ and $x_j$, but they also account for each particle $x_k$ in between $x_i$ and $x_j$. 

The latter two sums in \eqref{eq:dotrji} (note that the summands are the same) correspond to the single sum in \eqref{eq:dotri}. In fact, if the particles $i$ and $j$ have opposite sign, then the summand is similar to that in \eqref{eq:dotri}:
\begin{align*}
\left[b_j b_k f(r_{jk}) - b_i b_k f(r_{ik}) \right] = -b_i b_k g(r_{ik};r_{ji}). 
\end{align*}
We recall from Section \ref{s:rji:ri} that the sign of the summand is independent on the distance between the particles.

If $b_i$ and $b_j$ have the same sign, then the summand reads as
\begin{align*}
\left[b_j b_k f(r_{jk}) - b_i b_k f(r_{ik}) \right] = b_i b_k h(r_{ik};r_{ji}), 
\end{align*}
where for any parameter $\rho > 0$ the function $h$ is defined as
\begin{align*} 
h(\cdot ;\rho) : \R \setminus \{-\rho,0\} \to \R, \qquad 
h(r;\rho) := f(\rho+r) - f(r);  
\end{align*}
see Figure \ref{fig:h}. Note that the only difference between the expressions for $h$ and $g$ is the sign in front of $f(r)$. Hence, similar to $g(\cdot ; \rho)$, on $(-\infty,-\rho)$ and $(0,\infty)$ the function $h(\cdot ;\rho)$ has a sign and is monotone. The convexity of $f$ yields an additional bound on $h$:
\begin{align}
  h(r;\rho) <  f'(\rho+r) \rho < 0. \label{eq:hbound}
\end{align}

\begin{figure}[ht]
\begin{center}
\includegraphics[width=8cm]{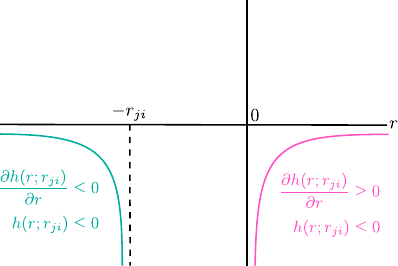}
\caption{Properties $h(r;r_{ji})$ for $r_{ji}$ fixed. The distance $r_{ik}$ represented by $r$ is contained in  $(-\infty,-r_i)$ or  $(0, \infty)$.  In the case $b_i = b_j$ the interaction between $r_{ik}$ and $r_{ji}$ is determined by $h$.  \label{fig:h} }
\end{center}
\end{figure}

Finally, Table \ref{table:gandh} summarizes the above by giving an overview of the contribution to $\dot r_{ji}$ of the summand in the third and fourth sum in \eqref{eq:dotrji} for both single particles $k = \kappa$ and for two neighboring particles $k = \kappa, \kappa+1$. Table \ref{table:gandh} will be the main tool in the results of the following section.	
	
\begingroup
\begin{table}[ht]
\centering
\begin{tabular}{cccrcc}
\toprule
particle(s) & \multicolumn{3}{c}{configuration} & \multicolumn{2}{c}{ $\dot{r}_{ji}$-contribution } \\
 \cmidrule(lr){2-4}  \cmidrule(lr){5-6} 
 & $b_i$ & $b_\kappa$ & position $\kappa$  & $b_i = -b_j$ & $b_i = b_j$ \\
\midrule
 $x_\kappa$   & + & + & $\kappa<i$ & $-$ & $-$ \\
$x_\kappa$   & + & + & $\kappa>j$ & + &  $-$ \\
$x_\kappa$   & + & $-$ & $\kappa<i$ & +  & +  \\
$x_\kappa$   & + & $-$ & $\kappa>j$ & $-$   &  +  \\
 $(x_\kappa,x_{\kappa+1})$ & + & $-$ & $\kappa+1< i$ & $-$ & $-$ \\
  $(x_\kappa,x_{\kappa+1})$ & + & $-$ & $\kappa> j$ & $-$ & +  \\
   $(x_\kappa,x_{\kappa+1})$ & + & +& $\kappa+1< i$ & + & + \\
  $(x_\kappa,x_{\kappa+1})$ & + & + & $\kappa> j$ & + & $-$  \\
 \bottomrule
\end{tabular}
\caption{Influence of a particle or a pair of particles on $\dot{r}_{ji}$. This influence depends on whether $b_i = -b_j$ or $b_i = b_j$.  The sign of the contribution is unchanged if the sign of both $b_i$ and $b_\kappa$ are swapped.  \label{table:gandh}}
\end{table}
\endgroup

\subsection{Estimates on $\dot{r}_{ji}$}
\label{s:rji:dotrji}

In this section we use Table \ref{table:gandh} to establish the key estimates on $\dot r_{ji}$ from \eqref{eq:dotrji}: in Lemma \ref{lem:ri} we prove a lower bound for when $j = i+1$ and in Lemmas \ref{lem:rjiodd} and \ref{lem:rjieven} we provide upper bounds for general $j$.

\begin{lem} $\dot r_i \geq - 2f(r_i) - C$ for all $i=1, \ldots, n-1$.
\label{lem:ri}
\end{lem}

\begin{proof} 
Starting from \eqref{eq:dotri} we apply $F_i(t) \geq - C$. We split the sum over $k$ in $k \leq i-1$ and $k \geq i+2$; see also \eqref{eq:dotrji}. Both sums can be treated analogously; we focus on the latter. If $b_i=1$ and $n-i$ is odd, then this sum can be written as a sum over the particle pairs 
$$
(x_{\kappa}, x_{\kappa+1})  \; \;  {\rm for \; \;}  \kappa=i+2, \, i+4, \, i+6, \ldots, n-1.
$$
Observe that  $b_\kappa=b_i$. From Table \ref{table:gandh} we see that each such pair yields a positive contribution to $\dot r_i$. Estimating these contributions from below by $0$, Lemma \ref{lem:ri} follows.  If $b_i =-1$, then $b_{\kappa} = b_i$ and the contribution to $\dot{r}_{i}$ is the same as the previous case $b_i = 1$. Hence, by a similar argument we may remove all terms from the sum. 

If $b_i=1$ and $n-i$ is even, then by a similar argument we may remove all terms from the sum over $k \geq i+2$, except for $k = n$. Since $b_n = b_i$, we see from Table \ref{table:gandh} that the summand at $k=n$ is positive, and thus Lemma \ref{lem:ri} follows. Finally, for $b_i=-1$ we can proceed by a similar argument.

\end{proof}

For the next lemmas we set 
$$r_0:= \infty, \qquad r_n:= \infty,$$ 
and use the convention $f(r_0) =0$ and $f(r_n) =0$. 
\begin{lem} For all $1 \leq i < j \leq n$, if $b_i = -b_j$, then
\begin{align*}
  \dot r_{ji} &\leq -2 f(r_{ji}) + 2 [f(r_j) + f(r_{i-1})]  + C.
\end{align*}
  \label{lem:rjiodd}
\end{lem}

\begin{proof} Starting from \eqref{eq:dotrji} we apply $F_{ji}(t) \leq C$. 
For the fourth sum in \eqref{eq:dotrji}, we note that the summand at $k = j+1$ equals 
\[
  b_j b_{j+1} f(r_j) + b_i b_{j+1} f(r_{{j+1},i})
  = f(r_j) + f(r_{{j+1},i}) \leq 2 f(r_j).
\]
From Table \ref{table:gandh} we see that each of the summands  corresponding to 
$$
(x_\kappa,x_{\kappa+1})   \; \;  {\rm for \; \;}  \kappa= j+2, j+4 , j+6, \ldots
$$
%
yields a negative contribution to the sum as $b_\kappa = b_{j}=-b_i$; we estimate these contributions from above by $0$. Finally, similar to the argument in the proof of Lemma \ref{lem:ri}, if the final term $k=n$ is not covered by the pairs above, then its contribution to the sum is negative. In that case we bound it from above by $0$. 

The third sum in \eqref{eq:dotrji} can be treated analogously; we obtain
\begin{align*}
 \sum_{k \leq i-1} \left[b_j b_k f(r_{jk}) + b_i b_k f(r_{ki}) \right] & \leq 2 f(r_{i-1}).
\end{align*}

For the first sum in \eqref{eq:dotrji} we note that the summand at $k = i$ equals $-f(r_{ji})$. The remaining number of summands is even. Grouping them as
\[
(x_\kappa)   \; \;  {\rm for \; \;}  \kappa = j-2, j-4, j-6 , \ldots
\]
the contribution of each term is negative. The second sum in \eqref{eq:dotrji} can be treated analogously; it is also bounded from above by $-f(r_{ji})$.
The lemma follows by putting all estimates together.
\end{proof}

\begin{lem} For all $1 \leq i < j \leq n$, if $b_i = b_j$, then
\[  \dot r_{ji} \leq r_{ji} f'(r_{ji}) + f(r_j) + f(r_{i-1}) + C.
\]
\label{lem:rjieven}
\end{lem}

\begin{proof} Again we start from \eqref{eq:dotrji} and apply $F_{ji}(t) \leq C$. 
For the fourth sum in \eqref{eq:dotrji}, note that the summand at $k = j+1$ equals 
\[
  b_j b_{j+1} f(r_{j,{j+1}}) - b_i b_{j+1} f(r_{i,{j+1}}) = - f(r_{j+1,{i+1}}) + f(r_{j}) \leq f(r_j).
\]
From Table \ref{table:gandh} we see that each of the summands corresponding to 
$$
(x_{\kappa},x_{\kappa+1})  \; \;  {\rm for \; \;}  \kappa = j+2, j+4, j+6 , \ldots
$$
yields a negative contribution to the sum (note that $b_\kappa =b_i$); we estimate these contributions from above by $0$. If $n-j$ is even then $k=n$ is not covered by the pairs above. In this case, $b_n = b_i$, hence the summand is negative and we bound it from above by $0$. 

Similarly, for the third sum of \eqref{eq:dotrji} we obtain
\[
  \sum_{k \leq i-1} [ b_j b_k f(r_{jk}) - b_i b_k f(r_{ik}) ] < f(r_i).
\]
 
Next we estimate the first sum of \eqref{eq:dotrji}. Since $b_i = b_j$ the number of terms is even. Grouping pairs together, we write it as
\begin{align}
  \sum_{k = i}^{j-1} b_j b_k f(r_{jk})
  = \sum_{\ell = 0}^{\frac{j-i}2 - 1} f(r_{j, i+2\ell}) - f(r_{j, i+2\ell+1})
  = \sum_{\ell = 0}^{\frac{j-i}2 - 1} h(r_{j, i+2\ell+1} ; r_{i+2\ell}). \label{eq:firstji}
\end{align}
Applying \eqref{eq:hbound} to \eqref{eq:firstji} and using that $f'$ is increasing we obtain 
\begin{align}
  \sum_{k = i}^{j-1} b_j b_k f(r_{jk}) 
  \leq \sum_{\ell = 0}^{\frac{j-i}2 - 1}  f'(r_{j, i + 2 \ell}) r_{i+2\ell}
  \leq f'(r_{ji}) \sum_{\ell = 0}^{\frac{j-i}2 - 1}  r_{i+2\ell}. \label{eq:rjieven1}
\end{align}
Similarly, we obtain for the second sum in \eqref{eq:dotrji} that
\begin{align}
\sum_{k = i+1}^{j} b_i b_k f(r_{ki})
\leq f'(r_{ji}) \sum_{\ell = 0}^{\frac{j-i}2 - 1}   r_{i+2\ell+1}. \label{eq:rjieven2}
\end{align}
Putting \eqref{eq:rjieven1} and \eqref{eq:rjieven2} together we obtain
\[
\sum_{k=i}^{j-1} b_j b_k f(r_{jk}) + \sum_{k=i+1}^{j} b_i b_k f(r_{ki})   \leq f'(r_{ji}) r_{ji}.
\]
Putting the estimates together we obtain the lemma.
\end{proof}

\section{Proof of the Main Theorem}
\label{s:pf}

The proof of Theorem \ref{theo:main} is given at the end of this section. First, we establish the key Lemma \ref{lem:rjimin} on which this proof relies. We consider the setting in Claim \ref{cl}. We combine Lemmas \ref{lem:ri}, \ref{lem:rjiodd} and \ref{lem:rjieven} to obtain bounds on $r_{ji}(t)$; see Lemmas \ref{lem:riest} and \ref{lem:rjibm} below. Note from the upper bound in \eqref{UB:ri} that we may assume that $r_i$ is as small as required by taking $\tau$ in Claim \ref{cl} small enough.

We start with two preparatory lemmas in which we exploit that $f(x) = \sgn(x) |x|^{-a}$:

\begin{lem} For $\tau$ small enough and for $i=1, \ldots, n-1$
\begin{align*}
r_i^a \dot{r}_{i} \geq -3 \quad \text{on }   \; \;[0, \tau).
\end{align*}
\label{lem:riest}
\end{lem}

\begin{proof} Using Lemma \ref{lem:ri} we obtain
\begin{align*}
\dot{r}_{i} \geq  \frac{-2}{r_i^a} - C.
\end{align*}
Taking $\tau$ small enough such that $C r_i^a \leq 1$ on $[0,\tau]$, Lemma \ref{lem:riest} follows. 
\end{proof}

For the next lemmas we recall that 
$r_0, r_n = \infty$ and that $r_0^{-a},r_n^{-a} =0$.

\begin{lem} Let $b := \min(1, \frac12 a)$. For $\tau$ small enough and for all $1 \leq i <j \leq n $
\[
\dot{r}_{ji} \leq   - \frac b{r_{ji}^a} + \frac{2}{r_j^a} + \frac{2}{r_{i-1}^a}
\qquad \text{on } [0, \tau).
\]
\label{lem:rjibm}
\end{lem}

\begin{proof} We first consider $b_i = - b_j$. We use Lemma \ref{lem:rjiodd} to obtain 
\[
\dot{r}_{ji} 
\leq  - \frac{2}{r_{ji}^a} + \frac{2}{r_j^a} + \frac{2}{r_{i-1}^a} + C.
\]
Taking $\tau$ small enough such that $C r_{ji}^a \leq 1$ on $[0,\tau]$, Lemma \ref{lem:rjibm} follows. 

For the other case $b_i=b_j$ the proof is similar: from Lemma \ref{lem:rjieven} we get
\[
\dot{r}_{ji} \leq  - \frac{a}{r_{ji}^a} + \frac{1}{r_j^a} + \frac{1}{r_{i-1}^a} + C,
\]
and then taking $\tau$ small enough such that $C r_{ji}^a \leq \frac a2$ on $[0,\tau]$, Lemma \ref{lem:rjibm} follows. 
\end{proof}

\begin{lem} Let $b$ be given as in Lemma \ref{lem:rjibm}. For $\tau$ small enough and for all  $t \in [0 , \tau]$ and  all $1 \leq i < j \leq n$ with $(j,i) \neq (n,1)$ we have
\[
r_{ji}(t) \geq c_0 \min \{r_{i-1}(t), r_j(t) \}, \qquad c_0 := \min\bigg\{ \Big( \frac b{16} \Big)^{\tfrac1a}, \Big( \frac b{12} \Big)^{\tfrac1{a+1}} \bigg\} > 0. 
\]
\label{lem:rjimin}
\end{lem}

\begin{proof} Fix $\tau > 0$ such that Lemmas \ref{lem:riest} and \ref{lem:rjibm} apply.
Lemma \ref{lem:rjimin} obviously holds at $t = \tau$. For $t \in [0,\tau)$ we reason by contradiction; suppose there exists $t_0 \in [0,\tau)$ such that
\[
  \alpha := r_{ji}(t_0) < c_0 \min \{r_{i-1}(t_0), r_j(t_0) \} =: c_0 \beta.
\]
Note that $\beta < \infty$ (indeed, while $r_0 = r_n = \infty$, the condition $(j,i) \neq (n,1)$ implies that $r_{i-1}(t_0)$ or $r_j(t_0)$ is finite).
Applying Lemma \ref{lem:riest} and integrating over $[t_0,t]$, we obtain
\[
\frac{1}{a+1}\left( r_k^{a+1} - (r_k^{\circ})^{a+1} \right)  \geq  -3t
\]
for $k \in \{i-1, j\}$. Rewriting this yields
\begin{align*}
r_k(t)  \geq \big( [\beta^{a+1}  - 3(a+1)(t-t_0)]_+ \big)^{\tfrac{1}{a+1}} 
\end{align*}
for all $t \in [t_0, \tau]$. Let 
\[
  t_1 := t_0 + \frac{\beta^{a+1}}{6(a+1)} 
\]
and note that 
\begin{align}
r_k  \geq 2^{-\tfrac{1}{a+1}} \beta > 0 \quad \text{on } \;\; [t_0,t_1]; \label{eq:ri_under:1}
\end{align}
thus $t_1 < \tau$.
Then, from Lemma \ref{lem:rjibm} and \eqref{eq:ri_under:1} we obtain 
\begin{align}
\dot{r}_{ji} 
&\leq  - \frac{b}{r_{ji}^a} + 2^{2 + \tfrac{a}{a+1} } \beta^{-a}
\leq - \frac{b}{r_{ji}^a} + \frac8{\beta^a} \quad \text{on } \;\; [t_0,t_1]. \label{drjiu}
\end{align}

We note from the following computation that the right-hand side in \eqref{drjiu} is negative initially at $t = t_0$:
\[
  \frac8{\beta^a} 
  < \frac{8 c_0^a }{\alpha^a} 
  \leq \frac{b}{2 r_{ji}(t_0)^a}.
\]
Moreover, the right-hand side in \eqref{drjiu} decreases as $r_{ji}$ decreases. Hence, we obtain
\begin{align*}
\dot{r}_{ji} 
\leq - \frac{b}{2 r_{ji}^a} \quad \text{on } \;\; [t_0,t_1]. 
\end{align*}
Multiplying by $r_{ji}^a$ and integrating over $[t_0, t]$, we obtain
\[
    \frac{1}{a+1}\left( r_{ji}(t)^{a+1} - \alpha^{a+1} \right)  \leq - \frac b2 (t - t_0).
\]
Rewriting this, we get
\[
  r_{ji}(t) \leq \Big[ \alpha^{a+1} - \frac b2 (a+1) (t - t_0) \Big]_+^{\tfrac{1}{a+1}}  \quad \text{on } \;\; [t_0,t_1].
\]
The first time $t_2$ at which the right-hand side hits $0$ is
\[
  t_2 
  = t_0 + \frac{2 }{ b (a+1)} \alpha^{a+1}
  < t_0 + \frac{2 }{ b (a+1)} c_0^{a+1} \beta^{a+1}
  \leq t_0 + \frac{\beta^{a+1} }{ 6 (a+1)} 
  = t_1 < \tau.
\]
Thus, $r_{ji}(t_2) = 0$ with $t_2 < \tau$, which contradicts $r_{ji} > 0$ on $[0,\tau)$.
\end{proof}

We obtain the desired lower bound on $r_i$ (recall \eqref{ri:LB}) from \eqref{LB:outer:xi} and Lemma \ref{lem:rjimin} by an iteration argument:

\begin{lem} For $\tau$ small enough there exists $c > 0$ such that $  r_i (t) \geq c (\tau - t)^{\tfrac1{1+a}}$ for all $i = 1,\ldots,n-1$   and all $t \in [0, \tau]$.
\label{lem:main} 
\end{lem}

\begin{proof} Take any $\tau > 0$ such that Lemma \ref{lem:rjimin} applies. Let $1 \leq i \leq n-1$. Since the statement is obvious for $t = \tau$, it is sufficient to consider $t \in [0,\tau)$. By \eqref{LB:outer:xi}, it is enough to show that 
\begin{equation} \label{pfzy} 
  r_i \geq c r_{n1}
\end{equation}
for some $c > 0$. Note that for $n=2$ we simply have $r_i = r_{n1}$, and thus we may assume $n \geq 3$.

We prove \eqref{pfzy} by induction over finitely many steps. The induction statement is as follows: if $r_i \geq c r_{jk}$ for some $c > 0$ and some integers $k,j$ satisfying $1 \leq k \leq i < j \leq n$ with $(j,k) \neq (n,1)$, then there exists $c' > 0$ such that either
\begin{subequations} \label{pfzx}
\begin{align} \label{pfzx1}
  j &\leq n-1 &&\text{and} & r_i &\geq c' r_{j+1,k}, \qquad \text{or} \\\label{prfz2}
  k &\geq 2 &&\text{and} & r_i &\geq c' r_{j,k-1}.
\end{align}
\end{subequations}
Observe that iterating the induction statement $(n-2)$-times yields \eqref{pfzy}. The conditions $(j,i) \neq (n,1)$, $j \leq n-1$ and $k \geq 2$ are of little importance; they simply ensure that the induction does not go beyond the end particles $x_1$ and $x_n$.

Initially, i.e.\ for $k = i$ and $j = i+1$, the condition in the induction statement is trivially satisfied with $c=1$. Let the condition in the induction statement be satisfied for some $c,j,k$. By Lemma \ref{lem:rjimin} we have
\begin{align} \label{pfzw}
  r_{jk} \geq c' \min \{r_{k-1}, r_j \}.
\end{align}
Suppose that the minimum is attained at $r_j$ (note that this implies $j \leq n-1$). Then by the induction statement
\[
  r_{j+1,k}
  = r_j + r_{jk}
  \leq \Big( \frac1{c'} + 1 \Big) r_{jk}
  \leq \frac1c \Big( \frac1{c'} + 1 \Big) r_i,
\] 
and thus \eqref{pfzx1} is satisfied. If the minimum in \eqref{pfzw} is instead attained at $r_{k-1}$, then a similar argument shows that \eqref{prfz2} holds. 
\end{proof}

\begin{cor} For $\tau$ small enough and all $i = 1,\ldots,n-1$ we have that $r_i \in C^{1/(1+a)}([0,\tau])$.
\label{cor:holder} 
\end{cor} 

\begin{proof} Take any $\tau > 0$ such that Lemma \ref{lem:main} applies. Let $1 \leq i \leq n-1$. Using Lemma \ref{lem:ri} and Lemma \ref{lem:main} we obtain
\begin{equation*}
  \dot r_i 
  \geq - \frac2{r_i^a} - C
  \geq \frac{- C'}{ (\tau - t)^{\frac{a}{a+1} }}.
\end{equation*} 
Then, using Lemma \ref{lem:rjiodd}  we obtain 
\begin{equation*}
  \dot r_i 
  \leq - \frac{2}{r_i^a} + \frac2{r_{i-1}^a} + \frac2{r_{i+1}^a} +  2C  
  \leq 0 + \frac{C'}{ (\tau - t)^{\frac{a}{a+1} }  } .
\end{equation*} 
For any $0 \leq s < t \leq \tau$ we obtain
\begin{align*}
  |r_i(t) - r_i(s)| 
  &\leq \int_s^t |\dot r_i (h)| \, dh 
  \leq \int_s^t \frac{C}{ (\tau - h)^{\frac{a}{a+1}} } \, dh
  = C'  \Big( (\tau - s)^{\tfrac{1}{a+1}} - (\tau - t)^{\tfrac{1}{a+1}} \Big) \\
  &= C'  \frac{t-s}{ (\tau - s)^\frac{a}{a+1} + (\tau - t)^\frac{a}{a+1}} 
  \leq C' \frac{t-s}{ (t - s)^\frac{a}{a+1}} 
  = C' (t - s)^{\tfrac{1}{a+1}}.
\end{align*}
\end{proof}

Finally, with Corollary \ref{cor:holder} we prove Theorem \ref{theo:main}:

\begin{proof}[Proof of Theorem \ref{theo:main}]
Consider the setting in Claim \ref{cl} and take $\tau > 0$ such that Corollary \ref{cor:holder} holds. It is left to prove that $x_i \in C^{1/(1+a)}([0,\tau])$ for all $i = 1,\ldots,n$.

Let $M = \sum_{i=1}^n x_i$. Consider the variable transformation from $(x_1, \ldots, x_n)$ to \\ $(M, r_1, r_2, \ldots, r_{n-1})$. Note that this transformation is linear and bijective. Hence, it is sufficient to show that $M, r_1, r_2, \ldots, r_{n-1} \in C^{1/(1+a)}([0,\tau])$. For the variables $r_i$ this is given by Corollary \ref{cor:holder}. For $M$ we compute from \eqref{ODE:xi} (note that $f$ is odd)
$
  \dot M = \sum_{i=1}^n F_i,
$
which is uniformly bounded on $(0,\tau)$. Hence, $M$ is Lipschitz continuous on $[0,\tau]$.
\end{proof}

\noindent \textbf{Acknowledgments:} PvM was supported by JSPS KAKENHI Grant Number JP20K14358. During this research Thomas de Jong was also affiliated to University of Groningen and Xiamen University. This research was partially supported by JST CREST grant number JPMJCR2014. 

\bibliographystyle{alpha}
\bibliography{refsPatrick,my_bib.bib}{}

\end{document}